\documentclass[a4paper,11pt,final,leqno,notitlepage]{article}

\usepackage{amsmath}
\usepackage{amsfonts}
\usepackage{amssymb}
\usepackage{amsthm} 

\usepackage{geometry}
\usepackage[blocks]{authblk}

\usepackage[T1]{fontenc}

\usepackage{fullpage}

\usepackage{enumerate}
\usepackage{enumitem}

\begin{document}
\renewcommand{\thepage}{\small\arabic{page}}
\renewcommand{\thefootnote}{(\arabic{footnote})}
\renewcommand{\thesection}{\arabic{chapter}.\arabic{section}}
\renewcommand{\thesubsection}{\arabic{subsection}}

\renewcommand\Affilfont{\small}

\author{Maria Trybu\l{}a}
\affil{Faculty of Mathematics and Informatics\\Adam Mickiewicz University in Pozna\'{n} \\ Uniwersytetu Pozna\'{n}skiego 4, 61-614 Pozna\'{n}, Poland\\ mtrybula@amu.edu.pl}
\affil{Institute of Mathematics and Informatics\\Bulgarian Academy of Sciences\\ Akad. Georgi Bonchev 8, 1113 Sofia, Bulgaria\\ maria.h.trybula@gmail.com}

\title{Multipliers on Spaces of Holomorphic Functions on Runge Domains in $\mathbb{C}^n$}
\maketitle

\begin{abstract}
We investigate multipliers on the space of holomorphic functions $H(\Omega)$, where $\Omega \subset \mathbb{C}^n$ is an open set. For Runge domains, we characterize these multipliers as convolutions with analytic functionals. Additionally, we explore Cartesian product domains, providing a representation of multipliers through germs of holomorphic functions. Finally, we identify the appropriate topology for analytic functionals, establishing a topological isomorphism with multipliers by utilizing the topology of uniform convergence on bounded sets inherited from the space of endomorphisms on $H(\Omega)$.
 \footnotetext[1]{


The research was supported by National Science Fund (Bulgaria), grant no. KP-06-N82/6.}

\end{abstract}
\newtheorem{thm}{Theorem}[subsection]
\newtheorem{lemma}[thm]{Lemma}
\newtheorem{proposition}[thm]{Proposition}
\newtheorem{corollary}[thm]{Corollary}
\newtheorem{remark}[thm]{Remark}
\newtheorem*{example}{Example}

\subsection{Introduction}
We investigate multipliers on the space of holomorphic functions $H(\Omega)$, where $\Omega\subset\mathbb{C}^n$ is a Runge domain. A multiplier is a linear continuous map $M:H(\Omega)\to H(\Omega)$ such that every monomial is an eigenvector of $M$. Throughout this paper, $\Omega$ denotes an arbitrary Runge domain in $\mathbb{C}^n$. Recall that a domain $\Omega$ is called Runge if every holomorphic function on $\Omega$ can be uniformly approximated on every compact subset of $\Omega$ by polynomials. We denote by $\mathsf{M}(\Omega)$ the set of all multipliers acting on $H(\Omega)$. The sequence of eigenvalues $(m_\alpha)_{\alpha\in\mathbb{N}^n}$ associated with a multiplier $M$ is referred to as its multiplier sequence.

Extensive research has been conducted on multipliers acting on spaces of holomorphic functions in one complex variable (see, e.g., \cite{B}, \cite{BM}, \cite{F}, \cite{M1}, \cite{M2}, \cite{MP1}, \cite{T1}, and the references therein). In particular, \cite{T2,T1} examined the one-variable case, and in this work, we develop a higher-dimensional generalization for Runge domains in $\mathbb{C}^n$. Multipliers on spaces of real analytic functions have garnered recent attention, with studies published in \cite{D1}, \cite{D2}. Taylor coefficient multipliers belong to a classical area of research, tracing back to Hadamard's early exploration of these operators in \cite{Hadamard}. The celebrated Hadamard multiplication theorem is, in fact, intertwined with multipliers (see the survey article \cite{R}).  

Our first major result is the characterization of all multipliers $\mathsf{M}(\Omega)$ in terms of analytic functionals (see Theorem 3.1). For product domains, i.e., when $\Omega$ is a Cartesian product of open planar sets, we establish a more refined description. Specifically, we express multipliers in terms of appropriate germs of holomorphic functions and examine the continuity of the relevant linear mappings (see Theorems 4.1, 4.2, and 4.4).   

Notation and auxiliary results are provided in Section 2.

\subsection{Preliminaries}
We begin by introducing basic concepts and notation that lay the groundwork for the rest of this study. Given a set $A\subset\mathbb{C}$, we define $A_*:=A\setminus\{0\}$. The symbols $\mathbb{N}$, $\mathbb{R}_{>0}$, $\mathbb{C}$, $\hat{\mathbb{C}}$, and $\mathcal{N}$ denote the set of natural numbers including zero, the set of positive real numbers, the complex plane, the Riemann sphere, and the set $\mathbb{C}^n\setminus\mathbb{C}_*^n$, respectively.

We denote by $\mathbb{D}(a,r)$ the open disk with center $a$ and radius $r$. When $a=0$, we write $\mathbb{D}(r)$. For $a\in\mathbb{C}^n$ and $r\in\mathbb{R}_{>0}^n$, the polydisc with center $a$ and polyradius $r$ is given by  
\[
\mathbb{P}_n(a,r) := \mathbb{D}(a_1,r_1)\times\ldots\times\mathbb{D}(a_n,r_n).
\]

For an open set of the form $\Omega = \Omega_1 \times \dots \times \Omega_n$ with each $\Omega_j\subset\mathbb{C}$, the distinguished boundary is defined as  
\[
\partial_0\Omega := \partial\Omega_1\times\ldots\times\partial\Omega_n.
\]
If each $\partial\Omega_j$ is piecewise $\mathcal{C}^1$ and a function $f$ is continuous on $\partial_0\Omega$, we define the integral
\[
\int_{\partial_0\Omega}\frac{f(\zeta)}{\zeta-z}d\zeta:=\int_{\partial\Omega_1}\ldots\int_{\partial\Omega_n}\frac{f(\zeta_1,\ldots,\zeta_n)}{(\zeta_1-z_1)\cdot\ldots\cdot(\zeta_n-z_n)}d\zeta_1\ldots d\zeta_n, \quad z\in\Omega.
\]

Throughout this paper, we use coordinatewise multiplication, denoted by  
\[
z\cdot w = zw := (z_1w_1,\ldots,z_nw_n),
\]
where $z = (z_1,\ldots,z_n)$ and $w = (w_1,\ldots,w_n) \in \mathbb{C}^n$. The identity element in this multiplication, denoted $(1,\ldots,1)$ ($n$ times), is written as $\mathbf{1}_n$. For $z\in\mathbb{C}_*^n$, we define $z^{-1}$ as $(1/z_1,\ldots,1/z_n)$. Given two sets $A, B\subset\mathbb{C}^n$, their algebraic product is given by  
\[
AB = A\cdot B := \{ ab : a\in A, b\in B \}.
\]
Similarly, coordinatewise addition is defined as  
\[
z + w := (z_1+w_1,\ldots,z_n+w_n),
\]
where $z, w\in\mathbb{C}^n$.  

For any $z\in\mathbb{C}^n$ and a function $f$, we define the dilatation of $f$ by  
\[
f_z(w) := f(zw).
\]

Here, $\alpha\in\mathbb{N}^n$ represents a multi-index. The symbols $\alpha!$, $|\alpha|$, and $\zeta^\alpha$ denote, respectively, the factorial product $\alpha_1!\cdot\ldots\cdot\alpha_n!$, the sum $\alpha_1+\ldots+\alpha_n$, and the monomial $\zeta_1^{\alpha_1}\cdot\ldots\cdot\zeta_n^{\alpha_n}$. The differential operator $D^\alpha$ is defined as  
\[
D^\alpha := \bigg(\frac{\partial}{\partial z_1}\bigg)^{\alpha_1} \circ \ldots \circ \bigg(\frac{\partial}{\partial z_n}\bigg)^{\alpha_n}.
\]

For a compact set $K$, we define the supremum norm of a bounded function $f:A\to\mathbb{C}$ as  
\[
\lVert f\rVert_K := \sup_{z\in K} |f(z)|.
\]
The space of continuous functions on $K$, denoted by $\mathcal{C}(K)$, is equipped with the norm $\lVert \cdot \rVert_K$.  

The space $H(\Omega)$ consists of holomorphic functions on an open set $\Omega\subset\hat{\mathbb{C}}^n$, endowed with the topology of uniform convergence on compact sets. For an arbitrary set $S\subset\hat{\mathbb{C}}^n$, the space of germs of holomorphic functions on $S$ is defined as  
\[
H(S) := \bigcup_{\omega\supset S\ \textup{open}} H(\omega).
\]
The topology on $H(S)$ is the finest locally convex topology for which all restriction maps  
\[
H(\omega)\owns f\mapsto f|_{S}\in H(S)
\]
are continuous. The subspace $H_0(S)$ consists of all functions in $H(S)$ that vanish at any infinite point of $S$. We denote by $\mathcal{P}_n$ the set of all polynomials in $n$ variables.

Let $\Omega\subset\mathbb{C}^n$ be an open set. A continuous linear functional on $H(\Omega)$ is referred to as an analytic functional on $\Omega$. The space of all analytic functionals, denoted by $H(\Omega)'$, forms the dual space of $H(\Omega)$. A compact set $K$ is called a carrier for $T\in H(\Omega)'$ if for every neighborhood $\omega$ of $K$ with $\omega\subset\Omega$, there exists a constant $C_\omega$ such that  
\[
|Tf| \leq C_\omega \rVert f\rVert_\omega, \quad f\in H(\Omega).
\]
In particular, if $T\in H(\Omega)'$ is carried by a compact set $K$, then $T(f)$ can be uniquely defined for any $f$ that is a uniform limit of functions in $H(\Omega)$ in some neighborhood of $K$.  

If $X$ and $Y$ are Fréchet spaces, the space of all continuous linear operators from $X$ to $Y$ is denoted by $\mathcal{L}(X,Y)$. If $X=Y$, we simply write $\mathcal{L}(X)$. Throughout this paper, the class of all multipliers on an open set $\Omega$ is denoted by $\mathsf{M}(\Omega)$.  
\newline For additional concepts in Functional Analysis or Complex Analysis not explicitly covered here, we refer to the books \cite{MV} and \cite{Krantz}.

\subsection{Representation of $\mathsf{M}(\Omega)$ via Analytic Functionals}

For a Runge open set $\Omega\subset\mathbb{C}^n$ we define
$$\mathcal{F}(\Omega):=\big{\{}T\in H(\mathbb{C}^n)^\prime\,:\,T \textup{ has a carrier contained in }z^{-1}\Omega\textup{ for every }z\in\Omega\setminus\mathcal{N}\big{\}}.$$

\begin{thm}\label{main theorem}
Suppose $\Omega\subset\mathbb{C}^n$ is a Runge domain. The map 
$$\Phi:\mathcal{F}(\Omega)\rightarrow \mathsf{M}(\Omega)$$
defined by
$$\Phi (T)(f)(z):=T_z (f_z ),\ \ \ f\in H(\Omega),\ z\in\Omega\setminus\mathcal{N},$$
acts as an isomorphism between vector spaces. Here $T_z$ is the unique element of $H(z^{-1}\Omega)^\prime$ satisfying $T=T_z$ on $H(\mathbb{C}^n)$, and $f_z(\cdot)=f(z\cdot).$ The multiplier sequence of $\Phi(T)$ coincides with the sequence of moments of the analytic functional $T$. The inverse mapping is defined as follows: for any $M\in\mathsf{M}(\Omega)$, the analytic functional 
$$T=\delta_{\mathbf{1}_n}\circ M:H(\mathbb{C}^n)\longrightarrow\mathbb{C}$$
extends to $H(z^{-1}\Omega)$ for every $z\in\Omega\setminus\mathcal{N}$ and satisfies the equation $\Phi(T)=M,$ where $\delta_{\mathbf{1}_n}$ denotes the point evaluation at $\mathbf{1}_n$. 
\end{thm}

\

\begin{proof}
We begin with two key observations. Fix $T\in\mathcal{F}(\Omega)$. 
First, we observe that
\begin{equation}\label{poczatek}
T(p_z) = \sum_{\alpha} p_\alpha T(\zeta^\alpha) z^\alpha,
\end{equation}
for every polynomial $p = \sum_{\alpha} p_\alpha \zeta^\alpha \in \mathcal{P}_n$ and $z\in\Omega\setminus\mathcal{N}$. Hence, the function
\[
g_{T}^{(p)}(w) := \sum_{\alpha} p_\alpha T(\zeta^\alpha) w^\alpha
\]
is well-defined and holomorphic on $\Omega$ for every $p\in\mathcal{P}_n$.

Moreover, note that
\begin{equation*}
T(f_{ab}) = T_z(f_{ab}) = T_z((f_a)_b),
\end{equation*}
for every $z\in\Omega\setminus\mathcal{N}, f\in H(\mathbb{C}^n),$ and $a, b\in\mathbb{C}^n$.
\newline We shall need

\begin{proposition}For every $T\in \mathcal{F}(\Omega)$ and $z\in\Omega$, there exists a neighborhood $U_z$  of $z$ such that  the operator
$$M_{z}:H(\Omega)\supset\mathcal{P}_n\owns p\mapsto g_{T}^{(p)}\in \mathcal{P}_n\subset H(U_z)$$
is continuous. In particular, $M_{z}$ extends to a continuous operator from $H(\Omega)$ to $H(U_z)$.
\end{proposition} 

\begin{proof} {\bf Case :} $\mathbf{z\in\Omega\setminus\mathcal{N}}$

We choose a carrier $K_z$ for $T_z$. Since $K_z$ is a compact subset of $z^{-1}\Omega$, there is $\epsilon_z\in\mathbb{R}_{>0}^n$ such that $\mathbb{P}_n({\mathbf{1}_n},3\epsilon_z)K_z\subset z^{-1}\Omega.$ $K_z$ is a carrier and therefore
$$|T_z(f_z)|\leq C_z\lVert f_z\rVert_{\mathbb{P}_n(\mathbf{1}_n,\epsilon_z)K_z}=C_z\rVert f\lVert_{z\mathbb{P}_n(\mathbf{1}_n,\epsilon_z)K_z}$$
for some $C_z>0$ and all $f\in H(\mathbb{C}^n).$ Consequently
\begin{equation}\label{5}
|T_z\big{(}(f_w)_z\big{)}|\leq C_z\lVert (f_w)_z\rVert_{\mathbb{P}_n({\mathbf{1}_n},\epsilon_z)K_z}=C_z\rVert f\lVert_{zw\mathbb{P}_n({\mathbf{1}_n},\epsilon_z)K_z}\leq C_z\rVert f\lVert_{z\mathbb{P}_n({\mathbf{1}_n},2\epsilon_z)K_z}
\end{equation}
for $f\in H(\mathbb{C}^n),\,w\in\mathbb{P}_n({\mathbf{1}_n},\delta_z),$ where $\delta_z\in\mathbb{R}_{>0}^n$ is so that 
$$\mathbb{P}_n({\mathbf{1}_n},\delta_z)\mathbb{P}_n({\mathbf{1}_n},\epsilon_z)\subset\mathbb{P}_n({\mathbf{1}_n},2\epsilon_z).$$
We set
\begin{equation*}
U_z:=z\mathbb{P}_n({\mathbf{1}_n},\delta_z)\ \ \textup{ for }\  z\in\Omega\setminus\mathcal{N}.
\end{equation*}
Clearly, 
\begin{equation}\label{inkluzja}
\Omega\setminus\mathcal{N}\subset\bigcup_{z\in\Omega\setminus\mathcal{N}}U_z.
\end{equation}
Assume $(p_k)\subset\mathcal{P}_n$ and $p_k\rightrightarrows 0$ locally uniformly on $\Omega$. 
 (\ref{5}) applied to $g_{T}^{(p_k)}$ gives
$$\lVert g_{T}^{(p_k)}\rVert_{U_z}\leq C_z\rVert p_k\lVert_{z\mathbb{P}_n({\mathbf{1}_n},2\epsilon_z)K_z}.$$
By the choice of $\epsilon_z$, we know that $z\mathbb{P}_n({\mathbf{1}_n},2\epsilon_z)K_z\Subset\Omega$. Hence, $g_{T}^{(p_k)}\rightrightarrows 0$ uniformly on $U_z$.

{\bf Case :} $\mathbf{z\in\Omega\cap\mathcal{N}}$

Take $r_z\in\mathbb{R}_{>0}^n$ such that $\mathbb{P}_n(z,2r_z)\Subset\Omega$ and $\partial_0 \mathbb{P}_n(z,2r_z)\cap\mathcal{N}=\varnothing$. 
We set
\begin{equation*}
U_z:=\mathbb{P}_n(z,r_z)\ \ \textup{ for }\  z\in\Omega\cap\mathcal{N}.
\end{equation*}

Since (\ref{inkluzja}), we might choose $N\in\mathbb{N},\,z_1,\ldots,z_N \in\partial_0 \mathbb{P}_n(z,2r_z)$ such that 
$\partial_0 \mathbb{P}_n(z,2r_z)\subset\bigcup_{j=1}^N U_{z_j}.$ 

Fix $(p_k)\subset\mathcal{P}_n$ and $p_k\rightrightarrows 0$ locally uniformly on $\Omega$.  Cauchy's integral formula applied to $g_{T}^{(p_k)}$ on  $\mathbb{P}_n(z,2r)$ gives 
$$g_{T}^{(p_k)}(w)=\frac{1}{(2\pi i)^n}\int_{\partial_0 \mathbb{P}_n(z,2r_z)} \frac{g_{T}^{(p_k)}(\zeta)}{\zeta-w}d\zeta,\ \ w\in\mathbb{P}_n(z,2r).$$
Consequently, applying once more (\ref{5}), we have 
\begin{multline}\label{klucz}
\sup_{w\in\mathbb{P}_n(z,r_z)}|\,g_{T}^{(p_k)}(w)|\ =\,\sup_{w\in\mathbb{P}_n(z,r_z)}\big{|}\frac{1}{(2\pi i)^n}\int_{\partial_0 \mathbb{P}_n(z,2r_z)} \frac{g_{T}^{(p_k)}(\zeta)}{\zeta-w} d\zeta\big{|}\ \leq \\ \big{(}\frac{2}{r_z}\big{)}^n\,\Vert g_{T}^{(p_k)}\rVert_{\partial_0 \mathbb{P}_n(z,2r_z)}\ \leq\ \big{(}\frac{2}{r_z}\big{)}^n \max_{j=1,\ldots, N}\,\lVert g_{T}^{(p_k)}\rVert_{U_{z_j}}\ \leq \\ \big{(}\frac{2}{r_z}\big{)}^n \max_{j=1,\ldots, N}\,\lVert p_k\rVert_{z_j\mathbb{P}_n({\mathbf{1}_n},2\epsilon_{z_j})K_{z_j}}.
\end{multline}
Clearly, $\bigcup_{j=1}^{N}z_j\mathbb{P}_n({\mathbf{1}_n},2\epsilon_{z_j})K_{z_j}\Subset\Omega$. Consequently, (\ref{klucz}) shows that $g_{T}^{(p_k)}\rightrightarrows 0$ uniformly on $\mathbb{P}_n(z,r_z)$. 
\end{proof}

For simplicity, we denote the extension in Proposition 3.2 using the same symbol:
$$M_z: H(\Omega)\rightarrow H(U_z),$$
$$M_z(f):=\lim_k M_z(p_k),\ \ f\in H(\Omega),$$
where $(p_k)\subset\mathcal{P}_n,\,p_k\rightrightarrows f$ locally uniformly on $\Omega$. Let us note that $M_z(f)$ is indeed holomorphic on $U_z$ due to the Weierstrass theorem. 
Furthermore, since for every $z,\,w\in\Omega$ and $p\in\mathcal{P}_n$ we have
$$M_{z}(p)=M_{w}(p)\textup{ on }U_z\cap U_w,$$
by continuity, we also have that for $f\in H(\Omega)$,
$$M_{z}(f)=M_{w}(f)\textup{ on }U_z\cap U_w.$$
This, coupled with the identity principle for holomorphic functions, proves that for every $T\in \mathcal{F}(\Omega),\,f\in H(\Omega)$ the expression:
\begin{equation*}
g^{(f)}:\Omega\owns z\longrightarrow M_{z}{(f)}(z)\in\mathbb{C}
\end{equation*}
defines a holomorphic function. Specifically, the mapping:
$$\Phi(T):H(\Omega)\owns f\mapsto g^{(f)}\in H(\Omega)$$
is well-defined.

Let us observe that every monomial is an eigenvector for $\Phi(T)$ immediately follows from its definition (see (\ref{poczatek})). Moreover, observe if we establish that $\Phi(T)$ is continuous, then we can conclude that $\Phi$ is injective. This is because $\Omega$ is a Runge domain. Therefore, $T\not=0$ if and only if there is $\alpha\in\mathbb{N}^n$ for which $T(\zeta^\alpha)\not=0$. To complete the argument, we need to demonstrate the continuity of $\Phi(T)$. This follows directly from the fact that every $M_z\in\mathcal{L}(H(\Omega),H(U_z))$, combined with the fact that $\bigcup_{z\in\Omega} U_z=\Omega$, where $U_z$ are as defined in the proof of Proposition 3.2 (compare with \cite[Step IV]{T1}).

So far, we have proved that $\Phi$ is a monomorphism. Therefore, we are only left with demonstrating its surjectivity before concluding the proof. This can be indicated as follows:
Fix $M\in\mathsf{M}(\Omega).$ Define 
$$Tf:=M(f)(\mathbf{1}_n),\ \ f\in H(\mathbb{C}^n).$$
Clearly, $T$ is a composition of 
$M$ and the evaluation at $\mathbf{1}_n$, both of which are continuous functions, implying that $T$ is continuous. For every $z\in\Omega\setminus\mathcal{N}$ there are a compact set $J_z\subset\Omega$ and a constant $C>0$ such that for $f\in H(\mathbb{C}^n)$ we have
\begin{equation}\label{3.9}
|T(f)|=|M(f)(\mathbf{1}_n)|=|M(f_{z^{-1}})(z)|\leq C\lVert f_{z^{-1}}\rVert_{J_z}=C\lVert f\rVert_{z^{-1} J_z}.
\end{equation} 
Clearly, (\ref{3.9}) remains true for every neighborhood of $z^{-1}J_z.$ Hence, $T$ has a carrier contained in $z^{-1}\Omega$, which implies that $T$ belongs to $\mathcal{F}(\Omega)$. Therefore, the expression $\widetilde{M}:=\Phi(T)$ makes sense. We claim that $M=\widetilde{M}$ on $\mathcal{P}_n$, and so everywhere. Indeed, for $\alpha\in\mathbb{N}^n,\,w\in \mathbb{C}^n\setminus\mathcal{N}$ we have
$$\widetilde{M}(\zeta^\alpha)(w)=T((w\zeta)^\alpha)=T(\zeta^\alpha)w^\alpha=M(\zeta^\alpha)(\mathbf{1}_n)w^\alpha=m_\alpha w^\alpha=M(\zeta^\alpha)(w).$$
This completes the proof.
\end{proof}

\begin{remark} \textup{In the one-dimensional case, we have a unique convex compact minimal carrier, but this does not hold for $n>1$ \textup{(}see \cite[Example pg. 314]{H}\textup{)}. For this reason, we have $\mathcal{F}(\Omega)=H(V(\Omega))^\prime$ presuming that $\Omega\subset\mathbb{C}$ is a convex domain, where $V(\Omega)$ is the so-called dilatation set of $\Omega$, that is, 
$V(\Omega)=\{z\in\mathbb{C}\,:\,z\Omega\subset\Omega\}$ (see \cite{T1}). Nonetheless, it is always true that $H(V(\Omega))^\prime\subset\mathcal{F}(\Omega)$.}
\end{remark}

\subsection{Special Case}

Throughout this section, we assume that $\Omega=\Omega_1\times\ldots\times\Omega_n$ with $\Omega_j\subset\mathbb{C}$ open for $j=1,\ldots,n$. Our first aim is to prove Theorems 4.1 and 4.2, which provide an alternative representation of multipliers on $H(\Omega)$ via suitable germs of holomorphic functions whose Laurent or Taylor coefficients coincide with the eigenvalues of the operator.

Every analytic functional $T\in H(\Omega)^\prime$ corresponds to a holomorphic function $f_T\in H_0((\hat{\mathbb{C}}\setminus \Omega_1)\times\ldots\times(\hat{\mathbb{C}}\setminus \Omega_n))$ via the so-called  K\"{o}the Grothendieck Tillmann duality (see \cite[Satz 1]{Tillmann}). Specifically, the function \(f_T\) associated with \(T\) is given by
\begin{equation}\label{2.2}
f_T(\zeta)=T\Big{(}\frac{1}{\zeta_1-\cdot}\cdot\ldots\cdot\frac{1}{\zeta_n-\cdot}\Big{)}, \ \ \zeta=(\zeta_1,\ldots,\zeta_n)\in(\hat{\mathbb{C}}\setminus \Omega_1)\times\ldots\times(\hat{\mathbb{C}}\setminus \Omega_n).
\end{equation}
Conversely, for every $f\in H_0((\hat{\mathbb{C}}\setminus \Omega_1)\times\ldots\times(\hat{\mathbb{C}}\setminus \Omega_n))$ there exists an analytic functional $T\in H(\mathbb{C}^n)^\prime$ such that 
$f_T=f.$
$T$ can be easily obtained from $f$: if $f\in H_0(U_1\times\ldots\times U_n)$ for some neighborhood $U_j$ of $\hat{\mathbb{C}}\setminus \Omega_j,\,j=1,\ldots,n$, then
\begin{equation}\label{2.3}
T_f(h)=\frac{1}{(2\pi i)^n}\int_{\gamma_1\times\ldots\times\gamma_n} h(\zeta)f(\zeta)d\zeta,\ \ h\in H(\Omega),
\end{equation}
where $\gamma_j\subset\mathbb{C}\setminus U_j$ is a finite union of closed curves  such that the index $n(\gamma_j,z)=1$ for any $z\in \Omega_j,$ and $n(\gamma_j,z)=0$ if $z\notin U_j$. The integral in (\ref{2.3}) does not depend on the choice of $\gamma_j$'s.

For $z\in\hat{\mathbb{C}}^n$ and $\mathcal{S}\subset 2^{\hat{\mathbb{C}}^n}$ we define:
$$H(z,\mathcal{S}):=\left\{f\in H(\{z\})|\ \textup{for every }S\in\mathcal{S}\textup{ there is }f_S\in H(A)\textup{ so that }f_S\equiv f\textup{ around }z\right\},$$

$$H_0(z,\mathcal{S}):=\left\{f\in H(z,\mathcal{S})|\ f(z)=0\right\}.$$

\begin{thm}\label{Theorem 4.1}
The mapping
$$\Psi:\mathsf{M}(\Omega)\owns M\mapsto\psi_{M}\in H_0((\infty,\ldots,\infty),\mathcal{O}_\Omega),$$
defined by
$$\psi_M(w)=\sum_{\alpha\in\mathbb{N}^n}\frac{m_\alpha}{w^{\alpha+\mathbf{1}_n}},$$
with $M(\zeta^\alpha)=m_\alpha\zeta^\alpha,\ \alpha\in\mathbb{N}^n$, acts as an isomorphism between the algebra of multipliers on $H(\Omega)$ with composition of operators and the space of germs of holomorphic functions at $(\infty,\ldots,\infty)$ that extend holomorphically on every set of the form $(\hat{\mathbb{C}}\setminus z_1^{-1}\Omega_1)\times\ldots\times (\hat{\mathbb{C}}\setminus z_n^{-1}\Omega_n),\,z\in\Omega\setminus\mathcal{N}$ with multiplication of Laurent series at $(\infty,\ldots,\infty)$ defined by 
$$\sum_{\alpha\in\mathbb{N}^n}\frac{a_\alpha}{w^{\alpha+\mathbf{1}_n}}\ \hat{\ast}\,\sum_{\alpha\in\mathbb{N}^n}\frac{b_\alpha}{w^{\alpha+\mathbf{1}_n}}:=\sum_{\alpha\in\mathbb{N}^n}\frac{a_\alpha b_\alpha}{w^{\alpha+\mathbf{1}_n}}.$$
Here
$$\mathcal{O}_{\Omega}=\Big{\{}(\hat{\mathbb{C}}\setminus z_1^{-1}\Omega_1)\times\ldots\times (\hat{\mathbb{C}}\setminus z_n^{-1}\Omega_n)\Big{\}}_{z\in\Omega\setminus\mathcal{N}}.$$
The multiplier sequence of the given multiplier is equal the Laurent coefficients at $(\infty,\ldots,\infty)$ of the corresponding germ.

Moreover, the action of the multiplier $M_\psi\in\mathsf{M}(\Omega)$ associated with $\psi\in H_0(\infty,\mathcal{O}_\Omega)$ can be calculated explicitly as follows
 \begin{equation}\label{4.3}
 M_\psi(f)(z)=\frac{1}{(2\pi i)^n}\int_{\gamma_1\times\ldots\times\gamma_n} f_z(\zeta)\psi_{(z)}(\zeta)d\zeta,\ \ f\in H(\Omega),
 \end{equation}
 where every $\gamma_j$ is a finite union of curves chosen for $\psi_{(z)},$ an extension of $\psi$ on a neighbourhood of $(\hat{\mathbb{C}}\setminus z_1^{-1}\Omega_1)\times\ldots\times (\hat{\mathbb{C}}\setminus z_n^{-1}\Omega_n)$, and $f_z.$ 
\end{thm}

\begin{thm}\label{Theorem 4.2}
The mapping 
$$\widehat{\Psi}:\mathsf{M}(\Omega)\owns M\mapsto\psi_{M}\in  H(0,\hat{\mathcal{O}}_\Omega),$$
where
$$\psi_M(w)=\sum_{\alpha\in\mathbb{N}^n}m_\alpha w^\alpha,$$
with $M(\zeta^\alpha)=m_\alpha\zeta^\alpha,\ \alpha\in\mathbb{N}^n$, acts as an isomorphism between the algebra of multipliers on $H(\Omega)$ with composition of operators
and the space of germs of holomorphic functions at $0$ that extend holomorphically on every set of the form $(\hat{\mathbb{C}}\setminus z_1\Omega_1^{-1})\times\ldots\times (\hat{\mathbb{C}}\setminus z_n\Omega_n^{-1}),\,z\in\Omega\setminus\mathcal{N}$ with multiplication of Taylor series at $0$, defined as
$$\sum_{\alpha\in\mathbb{N}^n}a_\alpha w^{\alpha}\,\ast\,\sum_{\alpha\in\mathbb{N}^n}b_\alpha w^{\alpha}:=\sum_{\alpha\in\mathbb{N}^n}a_\alpha b_\alpha w^{\alpha}.$$
Here 
$$\hat{\mathcal{O}}_{\Omega}:=\Big{\{}(\hat{\mathbb{C}}\setminus z_1\Omega_1^{-1})\times\ldots\times (\hat{\mathbb{C}}\setminus z_n\Omega_n^{-1}) \Big{\}}_{z\in\Omega\setminus\mathcal{N}}.$$
The multiplier sequence of the given multiplier equals the Taylor coefficients at $0$ of the corresponding germ.

Furthermore, the action of the multiplier $M_\psi\in\mathsf{M}(\Omega)$ associated with $\psi\in H(0,\hat{\mathcal{O}}_\Omega)$ can be calculated explicitly as follows
 \begin{equation}\label{4.3'}
 M_\psi(f)(z)=\big{(}\frac{-1}{2\pi i}\big{)}^n\int_{1/\gamma_1\times\ldots \times 1/\gamma_n} f\big{(}z\zeta^{-1}\big{)}\frac{\psi_{(z)}(\zeta)}{\zeta}d\zeta,\ \ f\in H(\Omega),
 \end{equation}
 where $\gamma_j$ is a finite union of curves chosen for $\big{(}w\rightarrow\frac{\psi_{(z)}(w^{-1})}{w}\big{)},\,\psi_{(z)}$ is an extension of $\psi$ on a neighbourhood of $(\hat{\mathbb{C}}\setminus z_1\Omega_1^{-1})\times\ldots\times (\hat{\mathbb{C}}\setminus z_n\Omega_n^{-1})$. 
\end{thm}

\begin{proof}[Proof of Theorem \ref{Theorem 4.1}]
The identification of $H_0((\infty,\ldots,\infty),\mathcal{O}_\Omega)$ with $\mathsf{M}(\Omega)$ as linear topological spaces follows from Theorem \ref{main theorem} and the K\"{o}the Grothendieck Tillmann duality.  
Indeed, for any $M\in\mathsf{M}(\Omega)$ and $z\in\Omega\setminus\mathcal{N}$, there are a function $\psi_{(z)}\in H_0((\hat{\mathbb{C}}\setminus z_1\Omega_1^{-1})\times\ldots\times (\hat{\mathbb{C}}\setminus z_n\Omega_n^{-1}))$ and a finite union of curves $\gamma_j\subset z_j^{-1}\Omega_j$ for $j=1,\ldots,n$ such that
\begin{equation}\label{4.4}
T_zf=\frac{1}{(2\pi i)^n}\int_{\gamma_1\times\ldots\times\gamma_n} f(\zeta)\psi_{(z)}(\zeta)d\zeta,\ \ f\in H(z^{-1}\Omega),
\end{equation}
where \(T=T_M\) (as in Theorem~\ref{main theorem}). From (\ref{4.4}) it follows that if $(m_\alpha)_{\alpha\in\mathbb{N}^n}$ is the multiplier sequence of $M$, then
$$\psi_{(z)}(w)=\sum_{\alpha\in\mathbb{N}^n}\frac{m_\alpha}{w^{\alpha+\mathbf{1}_n}}$$
in a neighborhood of $(\infty,\ldots,\infty)$. Hence, the function 
$$\psi(w):=\sum_{\alpha\in\mathbb{N}^n}\frac{m_\alpha}{w^{\alpha+\mathbf{1}_n}}$$
belongs to $H_0((\infty,\ldots,\infty),\mathcal{O}_\Omega).$

To determine the action of the multiplier \(M_\psi\) associated with \(\psi\in H_0((\infty,\ldots,\infty),\mathcal{O}_\Omega)\), assume that \(\psi\) extends to \(\psi_{(z)}\in H(U)\) for some open set 
\[
U=U_1\times\cdots\times U_n\supset (\hat{\mathbb{C}}\setminus z_1^{-1}\Omega_1)\times\cdots\times (\hat{\mathbb{C}}\setminus z_n^{-1}\Omega_n)
\]
and fix \(z\in\Omega\setminus\mathcal{N}\). Let \(\gamma_j\) be a curve that separates \(\hat{\mathbb{C}}\setminus U_j\) from \(z^{-1}\Omega_j\) for every \(j=1,\ldots,n\). Then, by applying the functional in \eqref{4.4}, we deduce that
\[
\Biggl|\frac{1}{(2\pi i)^n}\int_{\gamma_1\times\cdots\times\gamma_n} f(\zeta) \psi_{(z)}(\zeta)d\zeta\Biggr|\le \frac{1}{(2\pi)^n}\lVert\psi_{(z)}\rVert_{\gamma_1\times\cdots\times\gamma_n}\lVert f\rVert_{\gamma_1\times\cdots\times\gamma_n}\prod_{j=1}^{n}l(\gamma_j),
\]
where \(l(\gamma_j)\) denotes the length of \(\gamma_j\), and
\[
\frac{1}{(2\pi i)^n}\int_{\partial_0\mathbb{P}_n(r)}\zeta^\alpha\psi_{(z)}(\zeta)d\zeta=m_\alpha,\quad \alpha\in\mathbb{N}^n,
\]
for sufficiently large \(r=(r_1,\ldots,r_n)\). Hence, the operator
\[
T_{\psi}:H(\mathbb{C}^n)\ni f\mapsto \frac{1}{(2\pi i)^n}\int_{\partial_0\mathbb{P}_n(r)}f(\zeta)\psi(\zeta)d\zeta
\]
defines an element of \(\mathcal{F}(\Omega)\). Define \(\Psi^{-1}(\psi)\) as \(\Phi(T_\psi)\) (with \(\Phi\) as in Theorem~\ref{main theorem}). Since the Laurent series of \(\Psi\bigl(\Phi(T_\psi)\bigr)\) near \((\infty,\ldots,\infty)\) is determined by the multiplier sequence of \(\Phi(T_\psi)\), we obtain \(\Psi\bigl(\Phi(T_\psi)\bigr)=\psi\).
Multipliers form an algebra under composition, and this operation corresponds to the pointwise multiplication of multiplier sequences. Consequently, within \(H_0((\infty,\ldots,\infty),\mathcal{O}_\Omega)\), multiplication is equivalent to the multiplication of moments, which reflects a coordinatewise multiplication of the Laurent coefficients.
\end{proof}

We now proceed to study the topological aspects of Theorem 3.1. Specifically, we investigate the topology on $\mathcal{F}(\Omega)$ under which the map $\Phi:\mathcal{F}(\Omega)\rightarrow \mathsf{M}(\Omega)$ becomes a topological isomorphism. Here, $\mathsf{M}(\Omega)$ is equipped with the topology of uniform convergence on bounded sets inherited from $\mathcal{L}(H(\Omega))$. Our methods are based on those presented in \cite{T1}. In some places, instead of giving a full explanation, we shall refer the Reader to \cite{T1}.

From now on, assume that the topology on \(H(\Omega)\) is given by the family of seminorms
\[
\{\lVert\cdot\rVert_K:K\in\mathcal{K}\},
\]
where \(\mathcal{K}\) is a fundamental family of compact subsets of \(\Omega\). Every \(K\in\mathcal{K}\) is of the form 
\[
K=K_1\times\cdots\times K_n,
\]
with every \(K_j\subset\mathbb{C}\) polynomially convex (i.e., \(\mathbb{C}\setminus K_j\) is connected). Additionally, we assume that \(1\in K\) and, if \(0\in\Omega\), then \(0\in\operatorname{int}K\).

Fix a compact $K\in\mathcal{K}.$ Define:
$$\mathcal{F}(\Omega,K):=\big{\{}T\in H(\mathbb{C}^n)'|\ T\textup{ has a carrier in }z^{-1}\Omega\textup{ for every }z\in K\setminus\mathcal{N}\big{\}},$$
\begin{multline*}
H_0(\Omega,K):= \big{\{}f\in H_0(\{(\infty,\ldots,\infty)\})|\ f\textup{ extends on a neighborhood of } \\(\hat{\mathbb{C}}\setminus z_1^{-1}\Omega_1)\times\ldots\times (\hat{\mathbb{C}}\setminus z_n^{-1}\Omega_n) \textup{ for }z\in K\setminus\mathcal{N}\big{\}},
\end{multline*}
$$\mathsf{MC}(\Omega,K):=\big{\{}M\in\mathcal{L}\big{(}H(\Omega),\mathcal{C}(K)\big{)}|\ M\textup{ admits }  \\
\textup{all monomials as eigenvectors}\big{\}},$$
$$\mathsf{M}(\Omega,K):=\big{\{}M\in\mathcal{L}\big{(}H(\Omega),H(K)\big{)}|\ M\textup{ admits }  \\
\textup{all monomials as eigenvectors}\big{\}}.$$
The embedding $\iota:\mathsf{M}(\Omega,K)\hookrightarrow\mathsf{MC}(\Omega,K)$ is continuous, and as sets, we have \(\mathsf{M}(\Omega,K)=\mathsf{MC}(\Omega,K)\) (see the proof of \cite[Proposition 4.1]{T1}).

Before we formulate the last main result in this paper, we shall define a topology on $\mathcal{F}(\Omega,K)$.

Suppose $V=V_1\times\ldots\times V_n$, where every $V_j\subset\mathbb{C}$ is a nonempty, polynomially convex open set for $j=1,\ldots,n$. For any positive null-sequence $\delta=\big{(}\delta_k\big{)}_{k\in\mathbb{N}},$ i.e., $\delta_k\rightarrow 0^+,$ and  $f\in H((\hat{\mathbb{C}}\setminus V_1)\times\ldots\times(\hat{\mathbb{C}}\setminus V_n))$, define 
\begin{equation}\label{5.8}
|f|_{V,\delta}:=\max\Big{\{}\sup_{z\in \partial_0 V,\,\alpha\in\mathbb{N}^n}\frac{|D^\alpha f(z)|}{\alpha!}\delta_{(|\alpha|)},\sup_{\alpha\in\mathbb{N}^n}\ \frac{|D^\alpha(f\circ 1/\zeta)(0)|}{\alpha!}\delta_{(|\alpha|)}\Big{\}},
\end{equation}
where $\delta_{(k)}:=\delta_0\delta_1\cdot\ldots\cdot\delta_{k}.$ 

A straightforward computation shows that the function
\begin{equation}\label{5.9}
\upsilon_{K,\delta}(f):=\sup_{z\in K\setminus\mathcal{N}}\,\big{|}f_{(z)}\big{|}_{z^{-1}\Omega,\delta},\ \ f\in H_0(\Omega,K)
\end{equation}
defines a seminorm on $H_0(\Omega,K)$. Here,  $f_{(z)}\in H_0((\hat{\mathbb{C}}\setminus z_1^{-1}\Omega_1)\times\ldots\times (\hat{\mathbb{C}}\setminus z_n^{-1}\Omega_n))$ denotes an extension of $f$ to a neighborhood of $(\hat{\mathbb{C}}\setminus z_1^{-1}\Omega_1)\times\ldots\times (\hat{\mathbb{C}}\setminus z_n^{-1}\Omega_n)$. Hence, there exists a unique locally convex topology on $H_0(\Omega,K)$ for which the family of seminorms
$$\big{\{}\upsilon_{K,\delta}|\,\delta\,\textup{ is a strictly positive null-sequence}\big{\}}$$
is a fundamental system of seminorms. $H_0(\Omega,K)$ endowed with this topology is a DF-space. 
 In particular, $H_0(\Omega,K)$ has a web, is barrelled, and ultra-bornological. 
 
 Finally, for $T\in \mathcal{F}(\Omega,K)$, define
 \begin{equation}\label{5.11}
 |T|_{K,\delta}:=\upsilon_{K,\delta}(f_T)
 \end{equation}
 ($f_T$ is given by the K\"{o}the Grothendieck Tillmann duality). We equip $\mathcal{F}(\Omega,K)$ with the locally convex topology given by the system of seminorms $\big{\{}|\ \,|_{K,\delta}\,|\,\delta\textup{ is a strictly positive null-sequence}\big{\}}.$ We denote this topology as $\tau_{\Omega,K}.$

\begin{proposition}\label{Propozycja 5.1}
The map $\Phi:T\rightarrow M_T$ \textup{(}as in Theorem 3.1) defines a topological isomorphism between $\mathcal{F}(\Omega,K)$ and $\mathsf{M}(\Omega,K).$ Its inverse map is 
$\Theta:M\mapsto T_M,$ where $T_M(f)=M(f)(\mathbf{1}_n).$
Both maps send equicontinuous sets into  equicontinuous sets.

\end{proposition}

By going to the projective limit over $K\in\mathcal{K}$, as a direct consequence of Proposition \ref{Propozycja 5.1} we obtain: 
\begin{thm}\label{twierdzenie topologiczne}For every open set $\Omega=\Omega_1\times\ldots\times\Omega_n\subset\mathbb{C}^n,\,\Omega_j\subset\mathbb{C},\,j=1,\ldots,n$ the mapping
$$\Phi:\big{(}\mathcal{F}(\Omega),\textup{proj}_{K\Subset\Omega}\,(\mathcal{F}(\Omega,K),\tau_{\Omega,K})\big{)}\longrightarrow\big{(}\mathsf{M}(\Omega),\tau_b\big{)}$$
is a topological isomorphism.
\end{thm}

We are left to demonstrate the truth of Proposition \ref{Propozycja 5.1}.

\begin{proof}[Proof of Proposition \ref{Propozycja 5.1}]
First, we show that the map \(T\mapsto M_T\) sends equicontinuous sets into equicontinuous sets. From the proof of Theorem~\ref{main theorem}, we may assume that there exists a compact set \(L\subset\Omega\) and a constant \(C>0\) such that
\[
|Tf|\leq C\lVert f\rVert_{z^{-1}L}\quad \text{for every } z\in K\setminus\mathcal{N} \text{ and } f\in H(\mathbb{C}^n).
\] Consequently, 
$$\lVert M_Tf\rVert_{K}=\sup_{z\in K}\,|T_z(f_z)|\leq C\lVert f_z\rVert_{z^{-1}L}=C\lVert f\rVert_L.$$
Thus, \(T\mapsto M_T\) maps equicontinuous subsets of \(\mathcal{F}(\Omega,K)\) into equicontinuous (hence bounded) subsets of \(\mathsf{M}(\Omega,K)\). Since \(\mathsf{M}(\Omega,K)\) is barreled, every bounded set is equicontinuous, and hence \(T\mapsto M_T\) sends bounded sets into bounded sets. As \(\mathcal{F}(\Omega,K)\) is bornological, the map \(T\mapsto M_T\) is continuous from \(\mathcal{F}(\Omega,K)\) to \(\mathsf{M}(\Omega,K)\).
   
Now, suppose that \(\mathcal{M}\subset\mathsf{MC}(\Omega,K)\) is a nonempty family such that
$\lVert Mf\rVert_K\leq C\lVert f\rVert_U$ for some constant $C>0$ and an open set $U\Subset\Omega.$ Notice that
$$|T_Mf|=|M(f)(\mathbf{1}_n)|=|M(f_{z^{-1}})(z)|\leq C\lVert f_{z^{-1}}\rVert_U=C\lVert f\rVert_{z^{-1}U}$$
for all $f\in H(\mathbb{C}^n).$ 
This shows that the set \(\{T_M: M\in\mathcal{M}\}\) is equicontinuous in \(\mathcal{F}(\Omega,K)\).

To prove the continuity of $\Theta:M\mapsto T_M$, we must estimate the derivatives  of $f_T\in H_0(\Omega,K)$. Fix a positive null-sequence $\delta.$ For any $T\in\mathcal{F}(\Omega,K)$ we have:
\begin{multline*}
\sup_{\alpha\in\mathbb{N}^n}\,\frac{\big{|}D^\alpha\big{(}(f_T)_{(z)}\circ 1/\zeta\big{)}(0)\big{|}}{\alpha!}\,\delta_{(|\alpha|)}=\sup_{\alpha\in\mathbb{N}_*^n}\,|T(\zeta^{\alpha-\mathbf{1}_n})|\delta_{(|\alpha|)} \\
=\sup_{\alpha\in\mathbb{N}^n_*}\,|M_T\big{(}\delta_{(|\alpha|)}\zeta^{\alpha-\mathbf{1}_n}\big{)}(\mathbf{1}_n)|\leq\sup_{h\in\mathcal{S},\,w\in K}\,|M_T(h)(w)|,
\end{multline*}
where
$$\mathcal{S}:=\Big{\{}h_\alpha:\zeta\rightarrow\delta_{(|\alpha|+n)}\zeta^\alpha\big{|}\,\alpha\in\mathbb{N}^n\Big{\}}.$$
As $\delta_k\rightarrow 0^+$, for each compact $S$, there exists a constant $C>0$ such that $\delta_{(k)}\leq C\lVert\zeta\rVert_S^k.$ Therefore, $\mathcal{S}$ is bounded in $H(\mathbb{C}^n).$ 

The remaining part of $|(f_T)_{(z)}|_\delta$ we estimate as follows:
\begin{multline*}
\sup_{\alpha\in\mathbb{N}^n,\,w\in\partial_0(z^{-1}\Omega)}\,\frac{| D^\alpha(f_T)_{(z)}(w)|}{\alpha!}\,\delta_{(|\alpha|)}=\sup_{\alpha\in\mathbb{N}^n,\,w\in\partial_0(z^{-1}\Omega)}\,\big{|}\,_\zeta T_z\big{(}\frac{1}{(w-\zeta)^{\alpha+\mathbf{1}_n}}\big{)}\big{|}\,\delta_{(|\alpha|)} \\
=\sup_{\alpha\in\mathbb{N}^n,\,w\in\partial_0(z^{-1}\Omega)}\,\big{|}M_T\Big{(}\frac{\delta_{(|\alpha|)}}{(w-\frac{\zeta}{z})^{\alpha+\mathbf{1}_n}}\Big{)}(z)\big{|}\leq\sup_{h\in\mathcal{B},\,w\in K}\,\big{|}M_T(h)(w)\big{|},
\end{multline*}
where 
$$\mathcal{B}:=\Big{\{}h:\zeta\rightarrow\frac{\delta_{(|\alpha|)}}{(w-\frac{\zeta}{z})^{\alpha+\mathbf{1}_n}}\big{|}\,z\in K\setminus\mathcal{N},\,w\in\partial_0(z^{-1}\Omega),\,\alpha\in\mathbb{N}^n\Big{\}}.$$
Note that $\mathcal{B}$ is bounded in $H(\Omega)$ because if $J\Subset\Omega$, then
$$\sup_{\eta\in J,\,w\in\partial_0(z^{-1}\Omega),\,z\in K\setminus\mathcal{N},\,j=1,\ldots,n}\,|w_j-\eta_j/z_j|^{-1}\leq\textup{dist}^{-1}(J,\partial\Omega)\,\lVert\zeta\rVert_{K}<+\infty.$$

\end{proof}

{\bf Data Availability Statement.} All the relevant data have been included in the paper.

\vspace{0.5cm}
{\bf Declarations}

\vspace{0.1cm}
{\bf Conflict of interest.} The author states that there is no conflict of interest.

{\bf Contributions.} I am the only author of all statements contained in the paper.

\end{document}